\documentclass[12pt]{article}

\usepackage{mathtools}
\usepackage{amsmath,amssymb,amsfonts}        % AMS math packages
\usepackage[mathscr]{euscript}
\usepackage{dsfont}

\newtheorem{theorem}{Theorem}

\newtheorem{remark}[theorem]{Remark}

\newtheorem{definition}[theorem]{Definition}

\newcommand{\be}{\begin{equation}}
\newcommand{\ee}{\end{equation}}
\newcommand{\bes}{\begin{equation*}}
\newcommand{\ees}{\end{equation*}}
\newcommand{\beqa}{\begin{eqnarray*}}
\newcommand{\eeqa}{\end{eqnarray*}}

\newcommand{\R}{\mathbb R}

\makeatletter
\DeclareRobustCommand{\qed}{%
	\ifmmode % if math mode, assume display: omit penalty etc.
	\else \leavevmode\unskip\penalty9999 \hbox{}\nobreak\hfill
	\fi
	\quad\hbox{\qedsymbol}}
\newcommand{\openbox}{\leavevmode
	\hbox to.77778em{%
		\hfil\vrule
		\vbox to.675em{\hrule width.6em\vfil\hrule}%
		\vrule\hfil}}
\newcommand{\qedsymbol}{\openbox}
\newenvironment{proof}[1][\proofname]{\par
	\normalfont
	\topsep6\p@\@plus6\p@ \trivlist
	\item[\hskip\labelsep\itshape
	#1.]\ignorespaces
}{%
	\qed\endtrivlist
}
\newcommand{\proofname}{Proof}
\makeatother
\usepackage{amssymb}
\usepackage{amsmath}
\usepackage{algorithm}
\usepackage{algorithmic}
\usepackage{commath}
\usepackage{graphicx}
\usepackage{caption}
\usepackage{threeparttable}
\usepackage{soul}                 
\usepackage{subcaption}
\usepackage{lineno} 	  %  adds line numbers
\usepackage{hyperref}   %  compile twice. remove for arXiv
\hypersetup{
	colorlinks=false,
	linkcolor=blue,
	filecolor=magenta,      
	urlcolor=cyan,
}

\usepackage{color}

\begin{document}

%\begin{frontmatter}

%\title
\begin{center}
{\Large\bf Rational approximation and its application to improving deep learning classifiers}

%\author[FSET]
{V.~Peiris}, 
%\ead{vpeiris@swin.edu.au}
%\author[TelAviv]
{N.~Sharon},
%\ead{nsharon@tauex.tau.ac.il}
%\author[FSET]
{N.~Sukhorukova}
%\ead{nsukhorukova@swin.edu.au}
%\author[Deakin]
{J. Ugon}
%\ead{julien.ugon@deakin.edu.au}
\end{center}
%
%\address[FSET]{Faculty of Science, Engineering and Technology, Swinburne University of Technology, PO Box 218, Hawthorne, Victoria, Australia}
%\address[TelAviv]{Tel-Aviv University, Ramat-Aviv, Tel-Aviv, 69978, Israel} 
%
%\address[Deakin]{221 Burwood Hwy, Burwood, VIC 3125, Australia} 

\begin{abstract}
A rational approximation by a ratio of polynomial functions is a flexible alternative to polynomial approximation. In particular, rational functions exhibit accurate estimations to nonsmooth and non-Lipschitz functions, where polynomial approximations are not efficient.  We prove that the optimisation problems appearing in the best uniform rational approximation are quasiconvex, and show how to use this fact for calculating the best approximation in a fast and efficient method. The paper presents a theoretical study of the arising optimisation problems and provides results of several numerical experiments. In all our computations, the algorithms terminated at optimal solutions. We apply our approximation as a preprocess step to deep learning classifiers and demonstrate that the classification accuracy is significantly improved compared to the classification of the raw signals.
\end{abstract}

%\begin{keyword}



%\end{keyword}

%\end{frontmatter}

%\linenumbers      % if needed

\section{Introduction}\label{seq:intro}

Approximation by rational functions in Chebyshev norm (uniform approximation) attracts a considerable amount of research attention. One of the reasons for this is a large number of potential real-life applications. Another reason  is the  ability of rational functions to approximate abruptly changing functions efficiently. Rational function approximation can be regarded as an extension of well studied polynomial approximation~\cite{chebyshev, remez57, Remez69, NUG,Trefethen2012} just to name a few.  In particular, polynomials are simple  and easy to handle.  At the same time, polynomial approximations are not  efficient when one needs to approximate nonsmooth or non-Lipschitz  functions: high order polynomials are required, leading to severe oscillations and numerical instability.

One way to overcome this problem is to use piecewise polynomials or polynomial splines~\cite{NUG, NR, Schumaker1968, sukhorukovaoptimalityfixed}. This approach   is efficient when the location of the knots (that is, the location of switching from one polynomial to the next one) is known. In this case, the optimisation problem is convex and there exists a number of theoretical and  numerical techniques to deal with this problem. When the knots are free, the problem is much more complex~\cite{NUG, NR,Meinardus1989, Mulansky92, su10,SukUgonTrans2017, SukUgonCrouzeix2019} and therefore the computational ability of modern optimisation techniques to handle this problem is limited.

Rational approximation by a ratio of polynomials is a promising alternative to free knot spline approximation: the denominator polynomial allows one to produce accurate approximations to nonsmooth and non-Lipschitz functions. In particular, past studies shows that in the case of functions in $L_p~(1\leq p<\infty )$ the rational approximation and free-knot spline approximations are closely related, see~\cite[Chapter 8]{PetrushevPopov1987} and~\cite[Chapter 10]{lorentz1996constructive}.

Uniform approximation by rational functions was extensively studied in the 50s-60s of the twentieth century~\cite{Achieser1965, Rivlin1962, Boehm1964,Meinardus1967rational,Ralston1965Reme}.  This was mostly the period of theoretical research and the attempts to extend the application of the celebrated Remez method, originally developed for polynomial approximation~\cite{remez57,Remez69,Remez34}, to rational approximation.  The most recent developments in this area~\cite{Trefethen2018}  are dedicated to ``nearly optimal'' solutions, whose construction is based on Chebyshev polynomials.

In this paper we propose a new approach by looking at the problem from the point of view of optimisation. Namely, the optimisation problem is quasiconvex. This class of functions includes all the functions whose sublevel set is convex. In particular, any convex function is also quasiconvex.  There are a number of  theoretical studies dedicated to quasiconvex optimisation~\cite{DaniilidisHadjisavvasMartinezLegas2002, rubinov2013abstract, dutta2005abstract, SL}. The development of efficient  computational techniques is not  extensively studied. In this paper we propose a simple, but still efficient computational method that converges to a global minimum of the objective function. Our algorithm uses the fact that the problem is  quasiconvex and therefore it explores the structure of the objective function rather than assigning it to a class of general nonconvex functions.

The proposed approach has been tested on a number of functions, including nonsmooth and non-Lipschitz. Then this method was tested in application to deep learning. In particular, the raw data were substituted by the parameters of carefully constructed   approximation before training the networks. This approach enhances the classification accuracy. A similar approach was proposed in~\cite{Zamir2015, zamir2013optimization, Zamir2016}, where the authors substitute the raw data by the coefficients of piecewise polynomial based approximations and then apply a number of machine learning methods to test the classification improvement. The main drawback of these studies is the  assumption  that equidistant fixed knots approach is accurate and  therefore  convex models are efficient, while in the current study we use rational functions that may be considered as a suitable alternative to free knot approximation~\cite{PetrushevPopov1987}. Therefore, this study can be viewed as a step forward extension to~\cite{Zamir2015, zamir2013optimization, Zamir2016}.

This paper is organised as follows. Section~\ref{sec:motivation} provides an overall description of the problem and motivation.
In section~\ref{sec:formulation} we provide a mathematical formulation and extensive analysis of the problem as well as the results of numerical experiments with regards to the optimality of the obtained results.  Then, in  section~\ref{sec:application} we develop an optimisation-based  model to improve the classification accuracy in signal classification, using a number of deep learning models. The results are very promising and encouraging. Section~\ref{sec:conclusions} provides
the conclusions and identifies our further research directions.

\section{Motivation}\label{sec:motivation}

The goal of this paper is to develop a simple and efficient method for continuous function approximation by rational functions in Chebyshev norm. The choice of rational functions is based on a number of essential properties. 
\begin{enumerate}
\item Rational functions are simple to work with and much more flexible than classical and trigonometric polynomials when nonsmooth and non-Lipschitz functions are subject to approximation. 
\item Results from~\cite{PetrushevPopov1987, lorentz1996constructive} indicate that the approximation power of rational functions is equivalent to free knot polynomial spline approximation whose underlying optimisation problems are very complex and there is no efficient  computational tool for constructing optimal approximations~\cite{ NUG, FreeKnotsOpenProblem96}.  In particular, the problem of free knot approximation was listed as one of the most important open problem in approximation~\cite{FreeKnotsOpenProblem96}.
\item There are many potential applications for this study in data analysis, where raw data are approximated by suitable functions to enhance the efficiency~\cite{Zamir2015, SharonShkolnisky}.   
\end{enumerate}

A theoretical study of the optimisation problems appearing in this research is presented in section~\ref{sec:formulation}. In particular, it is demonstrated that the corresponding optimisation problem is  quasiconvex. The class of quasiconvex functions is rich, but  there are a number of efficient  tools and computational approaches to deal with these problems. This is not the case for general nonconvex functions (for example, approximation by free knot splines) and therefore this extra knowledge about the structure of the problem is very beneficial and valuable. 

This research has many potential applications. Some of them are quite theoretical~\cite{SharonShkolnisky}, aiming at evaluation non-analytic matrix functions, while some others~\cite{Zamir2015} are very applied, in particular, in the area of data analysis. It was demonstrated in~\cite{Zamir2015,Zamir2016sa} that the application of available machine learning methods to the parameters of  carefully designed approximations rather than raw signals significantly increase the classification accuracy. This observation motivated us to develop rational function-based models as a suitable alternative to free knot splines. Section~\ref{sec:application} studies this important practical application in details. 

\section{Constructing rational approximation} \label{sec:formulation} 

\subsection{Problem formulation}
Let $I$ be a closed segment on the real line and denote by $f(t) \in C^0\left(I\right)$ a function to be approximated. Let $\Pi_n$ be the space of polynomials of total degree $n$ and $\mathcal{R}_{n,m}$ be the set of all rational
functions of degrees $n$ and $m$, i.e.,
\[ \mathcal{R}_{n,m} = \{ p/q \, \mid \, p \in \Pi_n, \quad q \in \Pi_n \} .\]
Functions from $\mathcal{R}_{n,m}$ are also known as rational function of type $(n,m)$. The uniform rational best approximation problem over $I$, also known as the \textit{minimax} rational approximation, is defined as 
\begin{equation} \label{eqn:best_rational_app}
 \min_{r \in \mathcal{R}_{n,m}} \max_{x \in I} \abs{f(x) - r(x)} . 
\end{equation}
Classically, the problem is addressed by the Remez algorithm, based on the equioscillation characterization of the best approximation, see e.g.,~\cite[Chapter 13]{ralston1960mathematical}. Recently, rational approximation and best rational approximation have gained attention due to several improved algorithms and novel applications, see e.g.,~\cite{Filip2018, nakatsukasa2016computing}.  We obtain a solution to the best rational approximation via optimisation techniques, as described next.

Our optimisation problem reads:
\begin{equation}\label{eq:problem}
\min_{{\bf A}, \bf{B}}\sup_{t\in [c,d]}\left|f(t)-\frac{\bf{A}^T\bf{G(t)}}{\bf{B}^T\bf{H(t)}}\right|,
\end{equation}
subject to
\begin{equation}\label{eq:positivity}
\bf{B}^T\bf{H(t)}>0,~t\in I,
\end{equation}
where 
${\bf{A}}=(a_0,a_1,\dots,a_n)^T\in\R^{n+1}, {\bf{B}}=(b_0,b_1,\dots,b_m)^T\in\R^{m+1}$ are our decision variables, and ${\bf{G}}(t)=(g_0(t),\dots,g_n(t))^T$, ${\bf{H}}(t)=(h_0(t),\dots,h_m(t))^T$, where $g_j(t)$, $j=1,\dots,n$ and $h_i(t)$, $i=1,\dots, m$ are known functions. In the rest of the paper we refer to $g_j(t)$ and $h_i(t)$ as basis functions. Therefore, we construct the approximations in the form of the ratio of linear combinations of basis functions. Note that the constraint set is an open convex set.

In the case when all the basis functions are monomials, the approximations are rational functions from $\mathcal{R}_{n,m}$, and the problem is reduced to the uniform best rational approximation~\eqref{eqn:best_rational_app}. In this paper,  we are not restricted to rational functions: all the results are valid for any types of basis functions when~(\ref{eq:positivity}) is satisfied. For simplicity, we call these approximations rational approximations even when the numerator and/or denominator are not polynomials.

\begin{definition}\cite{SL,sion1958,JPCrouzeix1980quasi}
Function $ f(t)$ is quasiconvex if and only if its sublevel set 
$$ S_{\alpha}=\{x |f (x ) \leq \alpha \} $$
is  convex for any $\alpha\in\R$. 
\end{definition}
This definition is equivalent to the following one. 
\begin{definition}
A function $ f : D \rightarrow \R $ defined on a convex subset $D $ of a real vector space is called  quasiconvex if and only if  for any pair $x$ and $y$ from $D$  and $\lambda\in [ 0 , 1 ]$ one has 
$$ 
f (\lambda x + ( 1-\lambda ) y )
 \leq
 \max\{ f ( x ) , f ( y )\} .$$ 
\end{definition}
\begin{definition}
Function $f$ is quasiconcave if and only if $-f$ is quasiconvex.
\end{definition}
\begin{definition}
Functions that are quasiconvex and quasiconcave at the same time are called quasiaffine (sometimes quasilinear).
\end{definition}

\begin{theorem}
The objective function~(\ref{eq:problem}) is quasiconvex.
\end{theorem}
\begin{proof}
To complete the proof, we need the following two results, refer to~\cite{SL} for details.
\begin{enumerate}
\item The supremum of a family of quasiconvex functions is quasiconvex.
\item The ratio of two linear functions is quasiaffine.
\end{enumerate}

Note, first of all, that 
$$
v({\bf A},{\bf B})=\frac{\bf{A}^T\bf{G(t)}}{\bf{B}^T\bf{H(t)}}
$$
is the ratio of two linear functions (with respect to the polynomial coefficients) and therefore it is quasiaffine~\cite{SL}. Since 
$$
|w|=\max\{w,-w\},
$$
one can see that the objective function in~(\ref{eq:problem}) is the supremum of a family of quasiaffine (and therefore quasiconvex) functions. This completes the proof.
\end{proof}
\begin{remark}
 Note that the sum of two quasiconvex functions is not always quasiconvex. As a result, the least squares approximation approach leads to a complex optimisation problem.  Unlike polynomial and fixed knot spline approximation, there is no closed form solution for least squares approximation in the case of approximation by a ratio of linear combinations of basis functions, see e.g.,~\cite{hokanson2018least}.
\end{remark}

In some cases, the degree of the denominator and/or numerator can be reduced without loosing the accuracy, then one has
$$\frac{\sum_{j=0}^n a_jt_i^j}{\sum_{k=0}^{m}b_kt_i^k}=\frac{\sum_{j=0}^{n-\nu} a_jt_i^j}{\sum_{k=0}^{m-\mu}b_kt_i^k},$$
where $d=\min\{\nu,\mu\}$  is termed \textit{the defect} or the measure of degeneracy. Then, one can formulate the following necessary and sufficient optimality conditions~\cite{Achieser1965}.
\begin{theorem} \label{thm:equioscillation_characterization}
A rational function in $\mathcal{R}_{n,m}$ with defect $d$ is the best polynomial approximation of a function $f \in C^0(I)$ if and only if there exists a sequence of at least $n+m+2-d$ points of maximal deviation where the sign of the deviation at these points alternates.
\end{theorem}
%Note that when $d\neq 0$, many computational methods experience difficulties and instabilities. 
Since it is not easy to know the defect, the total number of alternating points for the optimal solution may not be a reliable stopping criterion. Nevertheless, for $d=0$ the number of alternating points is maximal and one can use this condition as a sufficient condition for validating optimality.

\subsection{Algorithm}
In our experiments, we discretise the approximation interval:
$$t_1,\dots,t_N\in I ,~N\gg \max\{m,n\}.$$ Then the problem is 
\begin{equation}\label{eq:obj}
\min z
\end{equation}
subject to
%\begin{equation}\label{eq:con1}
%f(t_i)-\frac{{\bf{A}}^T{\bf{G}}(t_i)}{{\bf{B}}^T{\bf{H}}(t_i)}\leq z,
%\end{equation}
%\begin{equation}\label{eq:con2}
%\frac{{\bf{A}}^T{\bf{G}}(t_i)}{{\bf{B}}^T{\bf{H}}(t_i)}-f(t_i)\leq z,
%\end{equation}
%\begin{equation}\label{eq:con3}
%{\bf{B}}^T{\bf{H}}(t_i)>0,
%\end{equation}
%\begin{equation}\label{eq:con4}
%t_i\in I , ~i=1,\dots,N.
%\end{equation}
\begin{align}
f(t_i)-\frac{{\bf{A}}^T{\bf{G}}(t_i)}{{\bf{B}}^T{\bf{H}}(t_i)}\leq z, \quad i=1,\dots,N \label{eq:con1} \\
\frac{{\bf{A}}^T{\bf{G}}(t_i)}{{\bf{B}}^T{\bf{H}}(t_i)}-f(t_i)\leq z, \quad i=1,\dots,N  \label{eq:con2} \\
{\bf{B}}^T{\bf{H}}(t_i)>0,\quad i=1,\dots,N  \label{eq:con3} 
\end{align}

Note that if $({\bf A},{\bf B})$ is an optimal solution, 
then $\alpha({\bf A},{\bf B}),~\alpha
\in 
\R
$ is also an optimal solution, since these two solutions represent the same fraction. One way to  avoid this ambiguity, is to fix one of the variables to~$1$ or $-1$ (the one that makes the denominator  ${\bf{B}}^T{\bf{H}}(t_i)$ positive). For computational purposes, one can substitute~(\ref{eq:con3}) by ${\bf{B}}^T{\bf H}(t_i)\geq \delta$, where $\delta$ is a small positive number. 

When $z$ is a variable, problem~(\ref{eq:obj})-(\ref{eq:con3}) is non-convex. If $z$ is fixed and the feasible set~(\ref{eq:con1})-(\ref{eq:con3}) has a feasible point, then the optimal solution is bounded by $z$. This feasibility can be verified by solving the following linear programming problem.
\begin{equation}\label{eq:obja}
\min \theta
\end{equation}
subject to
\begin{align}
(f(t_i)-z){{\bf{B}}^T{\bf{H}}(t_i)}-{{\bf{A}}^T{\bf{G}}(t_i)}\leq \theta,
\quad i=1,\dots,N  \label{eq:con1a} \\
{{\bf{A}}^T{\bf{G}}(t_i)}-(f(t_i)+z){{\bf{B}}^T{\bf{H}}(t_i)}\leq \theta, \quad i=1,\dots,N  \label{eq:con2a} \\
{\bf{B}}^T{\bf{H}}(t_i)\geq\delta,
\quad i=1,\dots,N  \label{eq:con3a} 
\end{align}
%\begin{equation}\label{eq:con1a}
%(f(t_i)-z){{\bf{B}}^T{\bf{H}}(t_i)}-{{\bf{A}}^T{\bf{G}}(t_i)}\leq \theta,
%\end{equation}
%\begin{equation}\label{eq:con2a}
%{{\bf{A}}^T{\bf{G}}(t_i)}-(f(t_i)+z){{\bf{B}}^T{\bf{H}}(t_i)}\leq \theta,
%\end{equation}
%\begin{equation}\label{eq:con3a}
%{\bf{B}}^T{\bf{H}}(t_i)\geq\delta,
%\end{equation}
%\begin{equation}\label{eq:con4a}
%t_i\in I , ~i=1,\dots,N.
%\end{equation}
If in the optimal solution of (\ref{eq:obja})-(\ref{eq:con3a}) $\theta\leq 0$, problem~(\ref{eq:obj})-(\ref{eq:con3}) is feasible otherwise it is infeasible.

We use a bisection method for quasiconvex optimisation. This method is simple and reliable and can be applied to any quasiconvex function (see~\cite{SL}, section~4.2.5 for more information). Our bisection method relies on solving the  feasibility problem, that is, finding if set~(\ref{eq:con1})-(\ref{eq:con3}) has a feasible point, which can be done by solving a linear programming problem. The bisection method is given next, in Algorithm~\ref{alg:bisection}.

\begin{algorithm}
	\caption{Bisection algorithm for quasiconvex optimisation} \label{alg:bisection}
	\begin{algorithmic}
		\REQUIRE Parameters for the non-convex problem~\eqref{eq:obj}-\eqref{eq:con3}. \\  \hspace{40pt} Absolute precision for maximal deviation $\varepsilon$
		\ENSURE Maximal deviation $z$ (within $\varepsilon$ precision) \\  \hspace{36pt}
		Coefficients $\bf{A},\bf{B}$ of the optimal linear combinations of~\eqref{eq:problem}
		\STATE set $l \leftarrow 0$
		\STATE set $u$ to be maximal deviation for a polynomial approximation
		\STATE $z \leftarrow (u+l)/2$
		\WHILE{$u-l \leq \varepsilon$}
		\STATE Check feasibility using  problem~\eqref{eq:obja}-\eqref{eq:con3a} with $z$.
		 \IF{feasible solution exists}
		\STATE $u \leftarrow z$
		\ELSE
		\STATE $l \leftarrow z$
		\ENDIF
		\STATE update $z \leftarrow (u+l)/2$
		\ENDWHILE 
       	\STATE $\bf{A},\bf{B} \leftarrow$ solve problem~\eqref{eq:obja}-\eqref{eq:con3a} with $z$
       	\RETURN $z,\bf{A},\bf{B} $	
	\end{algorithmic}
\end{algorithm}

\subsection{Numerical examples}
We demonstrate a few numerical aspects of our algorithm using the following examples. The code of this section is implemented in Matlab and is available online\footnote{\url{https://github.com/nirsharon/rational_approx}} for full reproducibility. 

In the first example, we consider a nonsmooth (and non-Lipschitz) function with a sharp, abrupt change,
\begin{equation} \label{eqn:function_example1}
 f(x)=\sqrt{\abs{x-0.25}},  \quad x \in [-1,1] . 
\end{equation}
The discretisation is done using the set of Chebyshev nodes $\{ \cos{\pi\frac{2\ell-1}{2n}} \}_{\ell=1}^N$, with $N$ sufficiently large and a precision $\varepsilon=10^{-14}$ for the bisection of Algorithm~\ref{alg:bisection} (more on this parameter in the sequel). For the rational approximation, we select the total degrees $n=4$ and $m=4$, which yields a (4,4) type rational approximant. To illustrate the advantage of rational approximation over polynomial approximation, we compare our minimax with the polynomial minimax of total degree $n+m =8$, calculated by the Remez algorithm, see, e.g.,~\cite{pachon2009barycentric}. The result is presented in Figure~\ref{fig:compare_approx}, where Figure~\ref{subfig:function_approx} shows the function $f$ of~\eqref{eqn:function_example1} side by side with its minimax polynomial approximation of degree $8$ and the minimax rational approximation of type $(4,4)$. 

\begin{figure}[ht]
	\begin{center}
	\begin{subfigure}{.49\textwidth}
		\includegraphics[width=\textwidth]{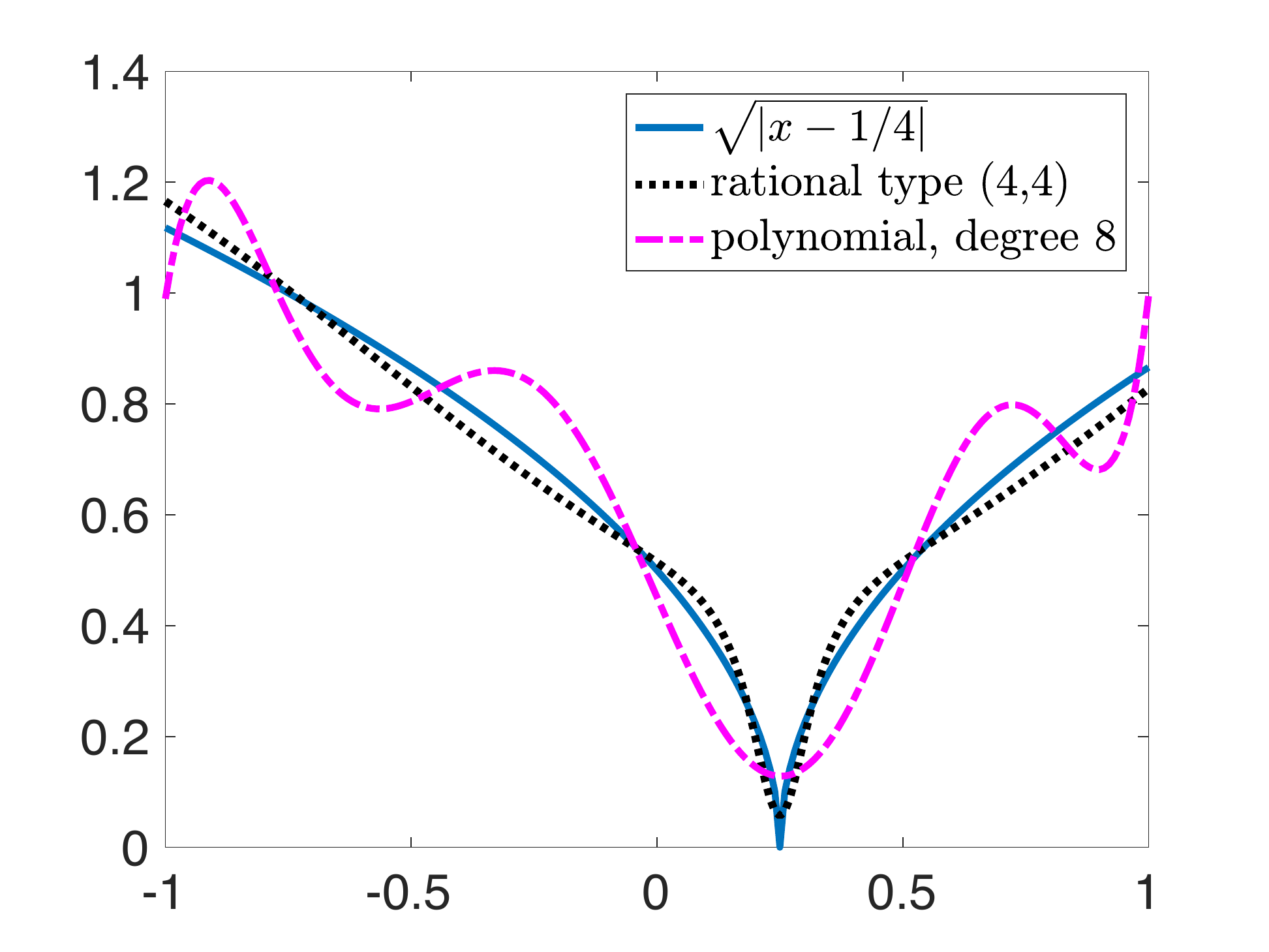}  
		\caption{The function and its approximations}
		\label{subfig:function_approx}
	\end{subfigure}
	\begin{subfigure}{.49\textwidth}
		\includegraphics[width=\textwidth]{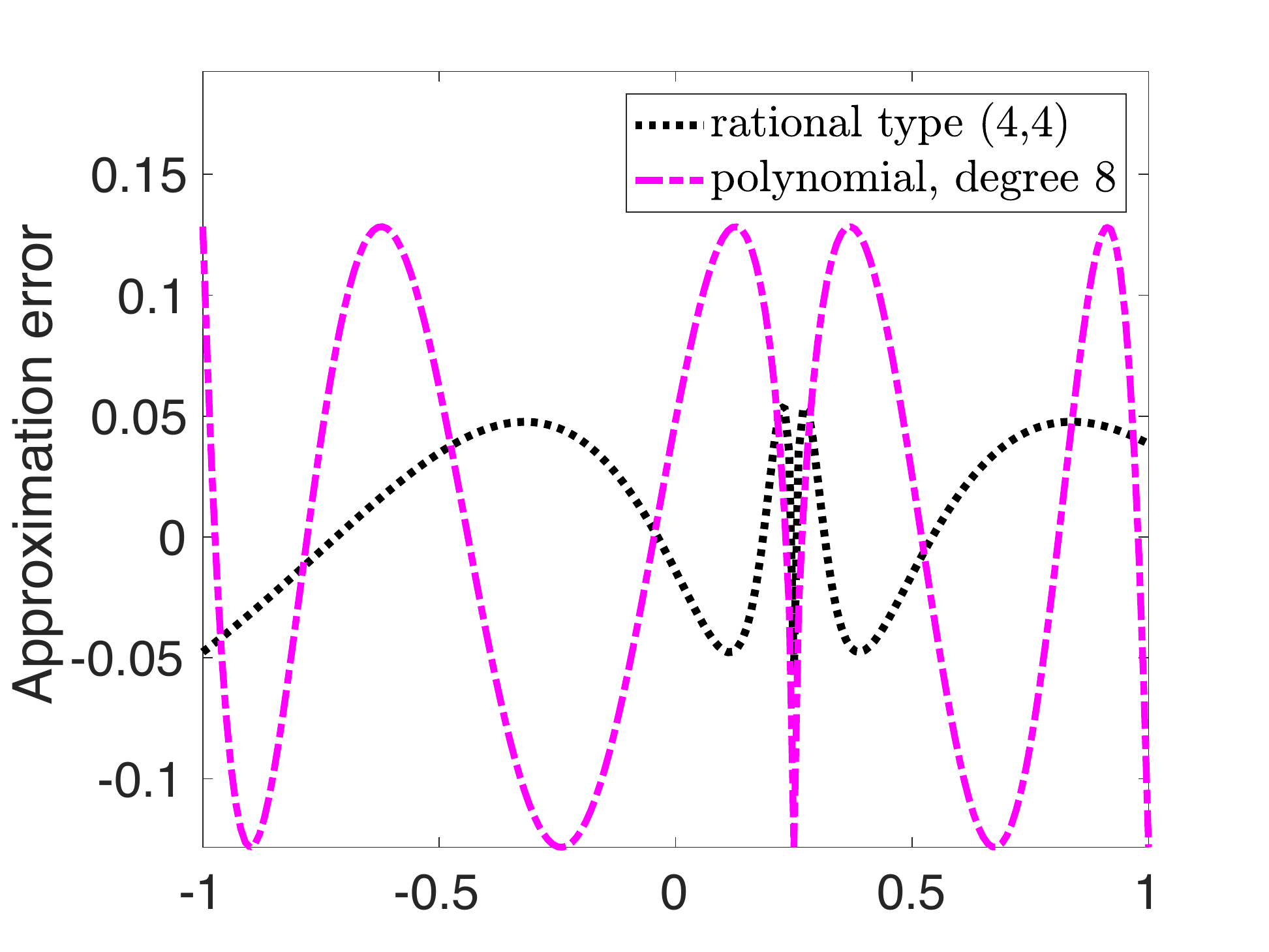}  
		\caption{Error curves}
		\label{subfig:error_curve}
	\end{subfigure}
	\caption{A comparison of best uniform approximations: rational versus polynomial. The polynomial is computed using the Remez algorithm. The rational minimax approximation is calculated via Algorithm~\ref{alg:bisection}.}
	\label{fig:compare_approx}
	\end{center}
\end{figure}

As seen in Figure~\ref{subfig:function_approx}, the rational function introduces a superior result over the polynomial one. In Figure~\ref{subfig:error_curve} we depict the error deviation of the two methods. The equioscillation property with a uniform magnitude of the error peaks (also for the polynomial minimax) are clearly seen. % and in particular, the eight peaks of the rational approximation \revns{suggest a defect value of $d=2$. However, if we decrease the tolerance  even more ($\varepsilon=10^{-14}$), two additional alternating points appear and therefore the optimality is verified by Theorem~\ref{thm:equioscillation_characterization}.
%}
% (one can also spot the $8+2$ peaks that are guaranteed by the polynomial version of Theorem~\ref{thm:equioscillation_characterization}).

In the second example, we study the role of the precision parameter of the bisection, that is $\varepsilon$ of Algorithm~\ref{alg:bisection}. This parameter determines the accuracy of the maximal deviation in~\eqref{eq:problem} that we obtain in our algorithm. A large $\varepsilon$ leads to a crude estimation of the maximal deviation, which in turn causes the optimisation problem~\eqref{eq:obja}-\eqref{eq:con3a} to generate a suboptimal solution to the uniform approximation problem~\eqref{eq:problem}. One possible validation for acquiring the optimal solution is via the equioscillation property, as described in Theorem~\ref{thm:equioscillation_characterization}, which will be demonstrate next. 

In this example, we use the same function of the previous one: $f$ of~\eqref{eqn:function_example1} but here the total degrees of our rational function are $m=n=3$. This choice leads to a defect free case ($d=0$) which means we expect to see $8 = m+n+2$ extreme, alternating values of the approximation error using the minimax. In Figure~\ref{fig:eps} we provide two error curves, generated by two different $\varepsilon$ values. In the first case, we set a relatively high value of $\varepsilon=0.1$, which results in the error curve of Figure~\ref{subfig:eps1}, where the error is not uniform and there are only $7$ extreme peaks. However, when we decrease $\varepsilon$, we derive a more accurate optimisation problem. Specifically, with $\varepsilon=10^{-5}$ we attain the error curve of Figure~\ref{subfig:eps2}, having $8$ error peaks that come with similar magnitude and alternating signs. This is a clear indication that the approximation is very close to be the minimax rational approximation.

\begin{figure}[ht]
	\begin{subfigure}{.49\textwidth}
		\centering
		\includegraphics[width=.95\textwidth]{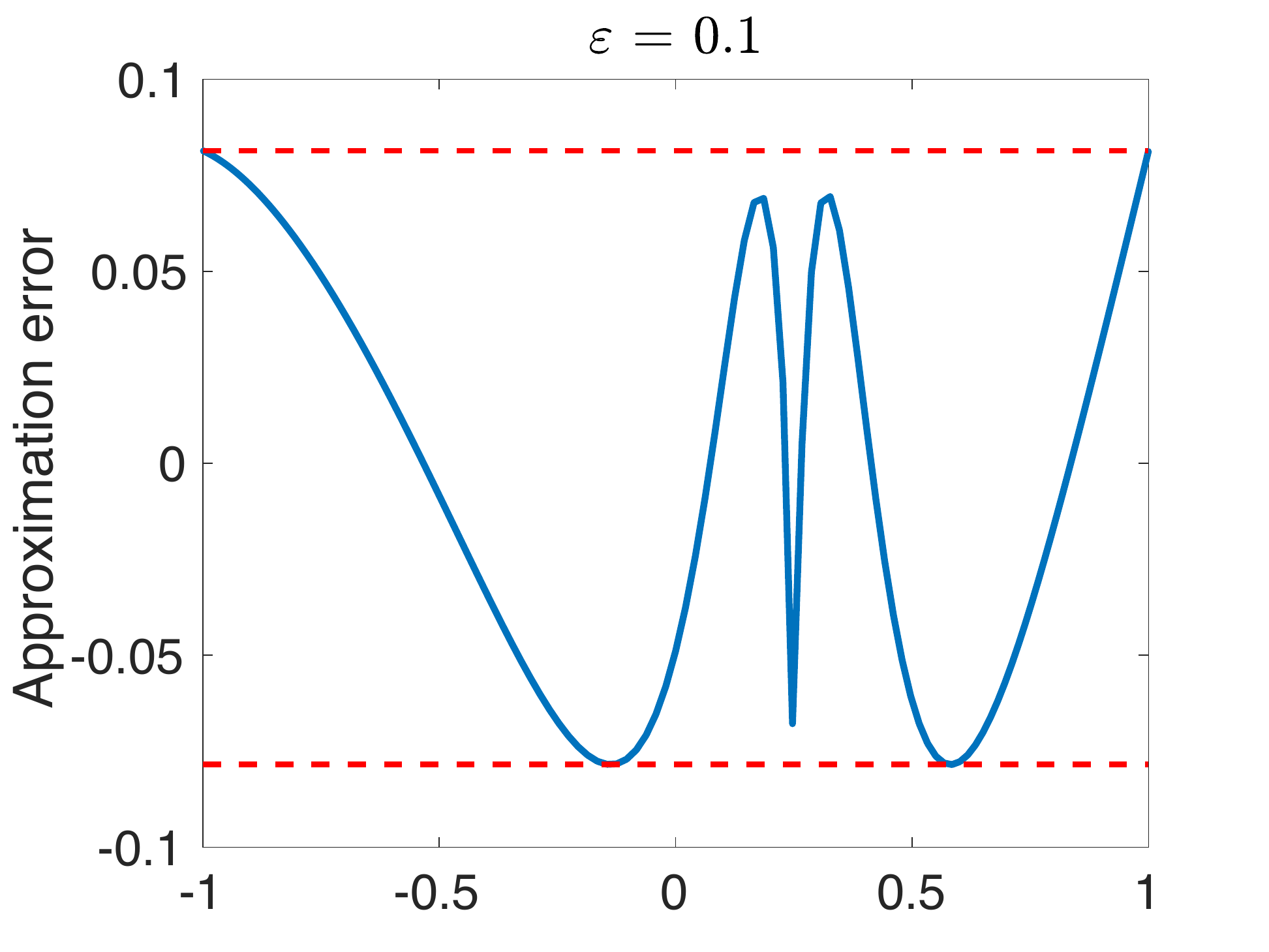}  
		\caption{}%A suboptimal solution with fewer oscillations and non-uniform error peaks}
		\label{subfig:eps1}
	\end{subfigure}
	\begin{subfigure}{.49\textwidth}
		\centering
		\includegraphics[width=.95\textwidth]{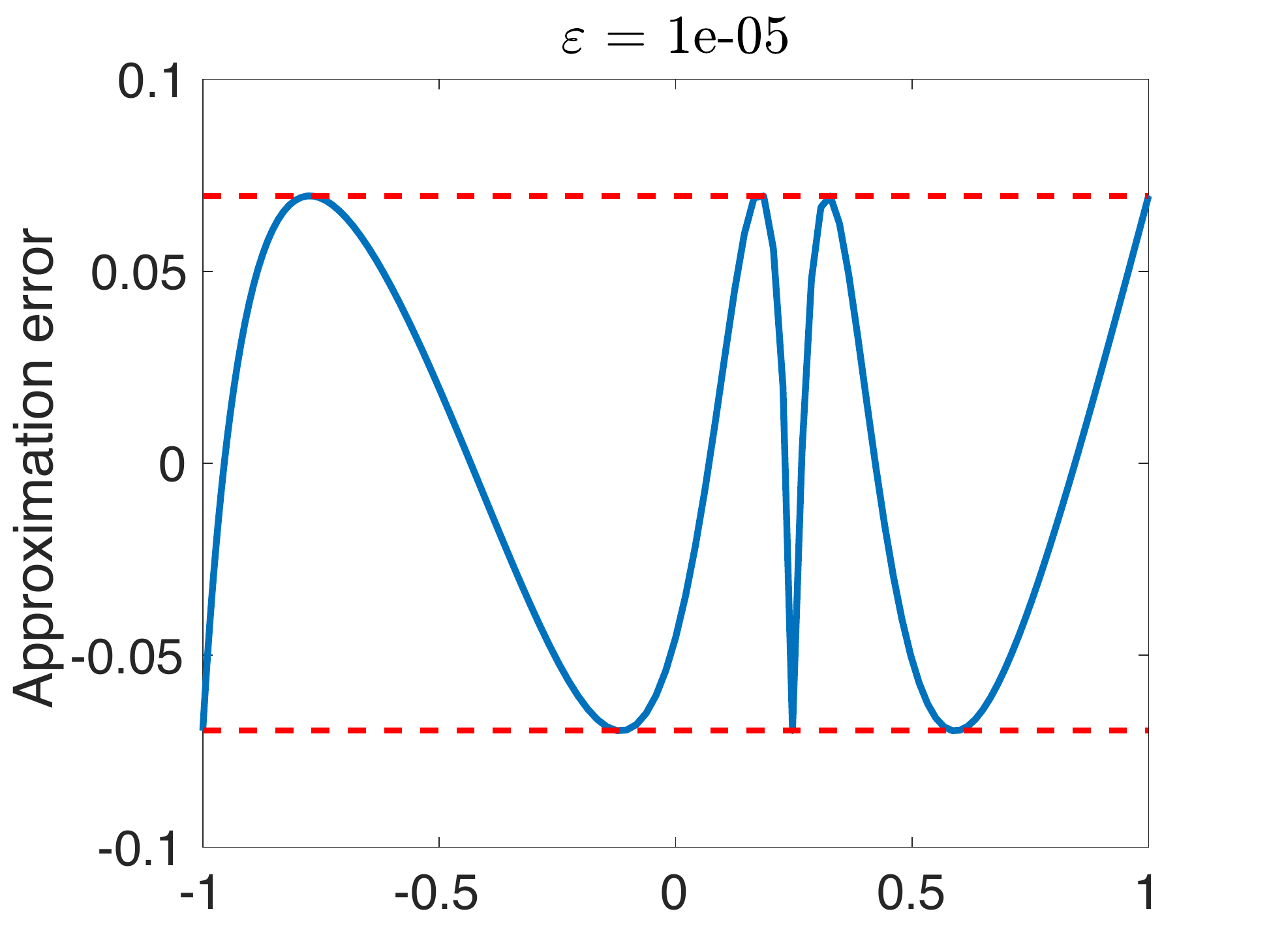}  
		\caption{}%\revns{A smaller $\varepsilon$ leads to an estimated minimax that satisfies }}
%{The optimal minimax}
		\label{subfig:eps2}
	\end{subfigure}
	\caption{The precision $\varepsilon$ of Algorithm~\ref{alg:bisection} and its influence on achieving the best uniform rational approximation: with smaller $\varepsilon$, the error curves tends to have more uniform error peaks with number of oscillations as described in Theorem~\ref{thm:equioscillation_characterization}}
	\label{fig:eps}
\end{figure}

\section{Application to deep learning}\label{sec:application}

\subsection{Data description and classification settings}\label{ssec:Data}

In this study, we use electroencephalogram (EEG) signals, which are recorded brain waves, from a publicly available dataset, collected by the epileptic center at the University of Bonn, Germany~\cite{andrzejak2001indications}. This dataset consists of five datasets  A, B, C, D, and E. Each set contains 100~signal segments recorded during 23.6~seconds with a sampling frequency of 173.61~Hz. Each signal has 4097~recordings. In this paper, we apply our approximation-based approach to develop an automatic procedure to distinguish between the first two sets (set A and B). Signal segments of sets A and B are collected from five healthy volunteers with eyes open and closed, respectively.

In our experiments, 75\% of each set is used as the corresponding training sets, while the remaining 25\% is the test sets. The reported classification accuracies correspond to the test sets. We evaluate the test accuracy obtained over raw data signals and compare them with the accuracy obtained over the parameters of the approximations.

As a classifier, we use deep learning tools from Python package scikit-learn (version 0.22.1)~\cite{scikit_learn}. We use up to three hidden layers, but since the classification results with one or two hidden layers are much better than the results with three hidden layers, we only report the former.

\subsection{Approximations}\label{ssec:Approximations}

In our experiments, we use two models.
\begin{enumerate}
\item[M1] Each data segment is approximated by the minimax rational function of type $(n,m)$, that is the solution of~(\ref{eq:problem}).
\item[M2] Each segment is approximated by a function \[ A\sin(\omega t+\tau),\]
where $A$ is the amplitude, constructed in the form of rational function. As before, the degree of the numerator  polynomial is~$n$ and the degree of the denominator polynomial is~$m$. Frequency~$\omega$ and  phase shift~$\tau$ are unknown.
\end{enumerate}

Model~M1 has been studied in Section~\ref{sec:formulation} and its application is straightforward, while Model~M2 requires more study. The corresponding optimisation problem is as follows:
\begin{equation}\label{eq:problemsin}
\min_{\substack{ {\bf A}, {\bf B} \\ \omega, \tau}}\sup_{t\in I}\left |f(t)-\frac{{\bf{A}}^T{\bf{G}}(t)}{{\bf{B}}^T{\bf{H}}(t)}\sin(\omega t+\tau)\right|,
\end{equation}
subject to
\begin{equation}\label{eq:positivitysin}
{\bf{B}}^T{\bf{H}}(t)>0,~t\in I.
\end{equation}
 This problem is not quasiconvex when $\omega$ and $\tau$ are unknowns, therefore we use the approach from~\cite{Zamir2015} to obtain $\omega$ and $\tau$: we choose a finite set of possible values for $\omega$ (natural numbers between $1$ and $15$) and $\tau$ (in our experiments  $0, \pi/4, \pi/2, 3\pi/4$) and run a brute force search in the form of double loops with fixed values for~$\omega$ and $\tau$. At each iteration of these loops, the problem is back to be quasiconvex --- the corresponding basis functions are obtained by multiplying monomials and a sine-function with known $\omega$ and $\tau$. 

Unlike the studies of K-complexes~\cite{Zamir2015, zamir2013optimization, Zamir2016}, there is no specific pattern to distinguish between these two classes. Figure~\ref{fig:class1} depicts the raw signal (the oscillatory blue) from the first class (Class~A) and its approximation (in red) for Model~M1  and Model~M2. Similarly, Figure~\ref{fig:class2} presents the raw signal (oscillatory blue) from the second class (Class~B) and its approximation (in red) for Model~M1 and Model~M2. It is clear from the pictures that Model~M2 is more accurate due to the additional sine component. 

\begin{figure}[ht]
	\begin{center}
		\begin{subfigure}{.49\textwidth}
			\includegraphics[width=\textwidth]{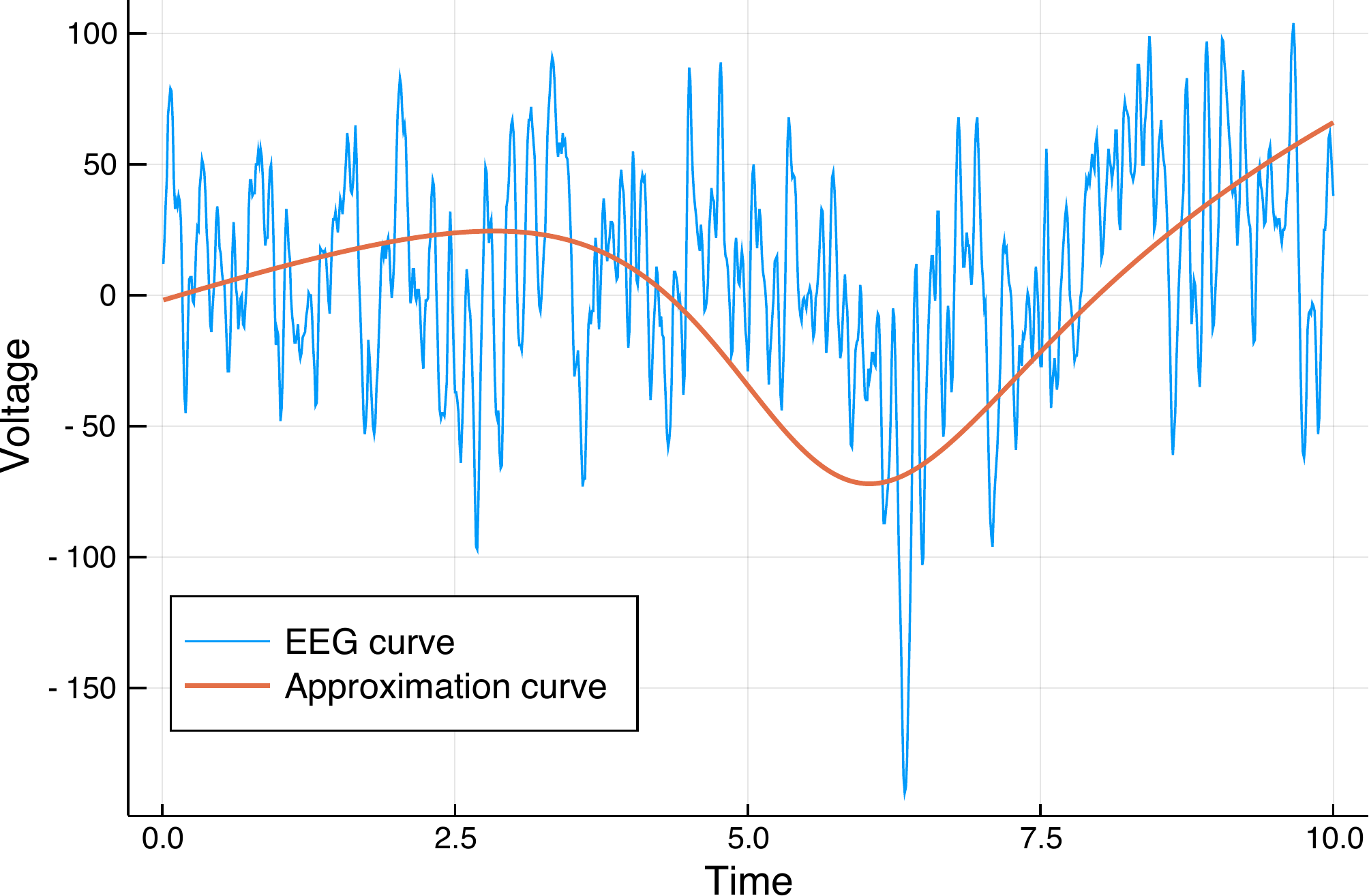} 
			\caption{Model~M1}
			\label{subfig:model1a}
		\end{subfigure}
		\begin{subfigure}{.49\textwidth}
			\includegraphics[width=\textwidth]{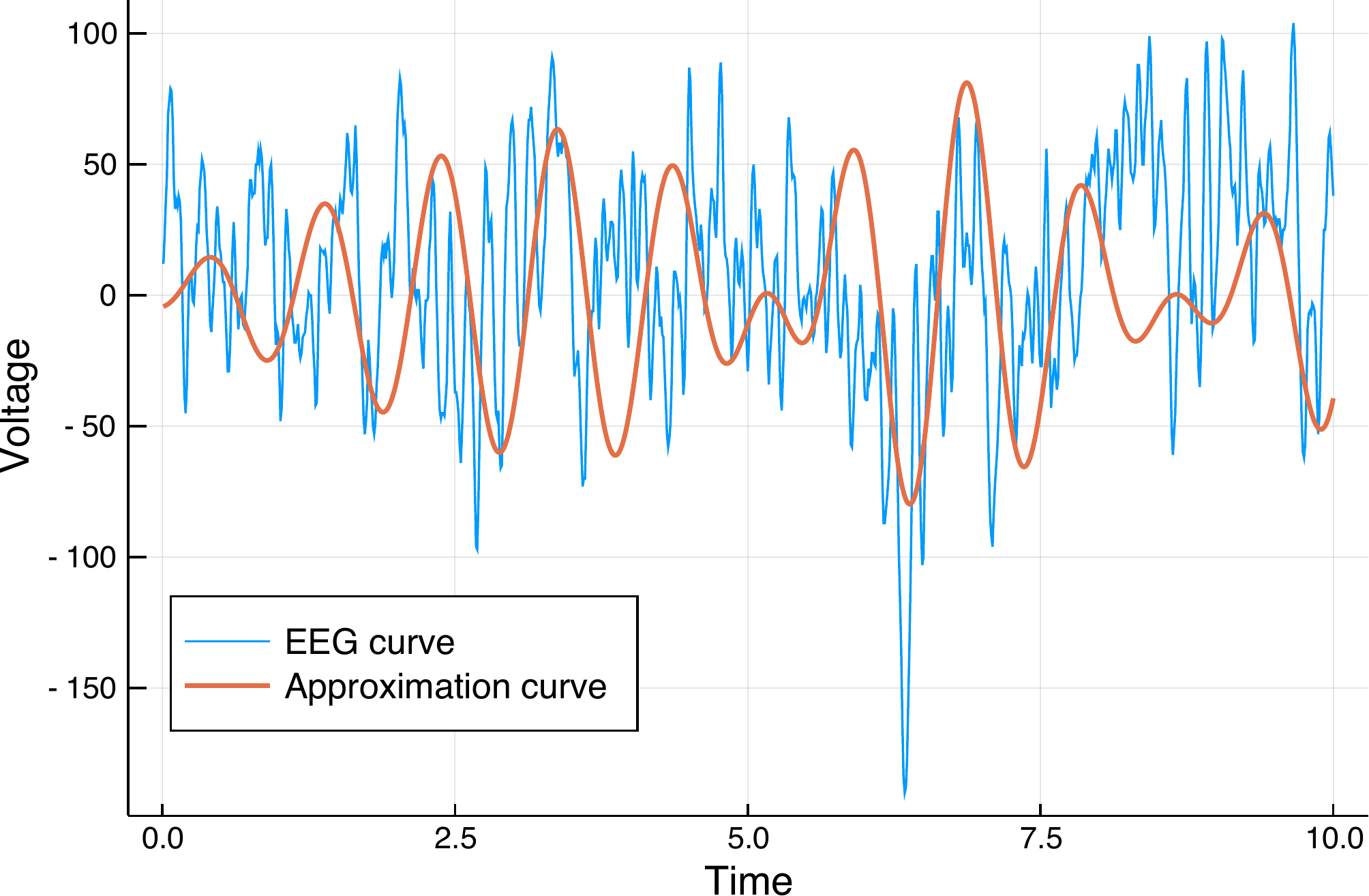}  
			\caption{Model~M2}
			\label{subfig:model2a}
		\end{subfigure}
		\caption{Signal from Class~A and its approximation}
		\label{fig:class1}
	\end{center}
\end{figure}

\begin{figure}[ht]
	\begin{center}
		\begin{subfigure}{.49\textwidth}
			\includegraphics[width=\textwidth]{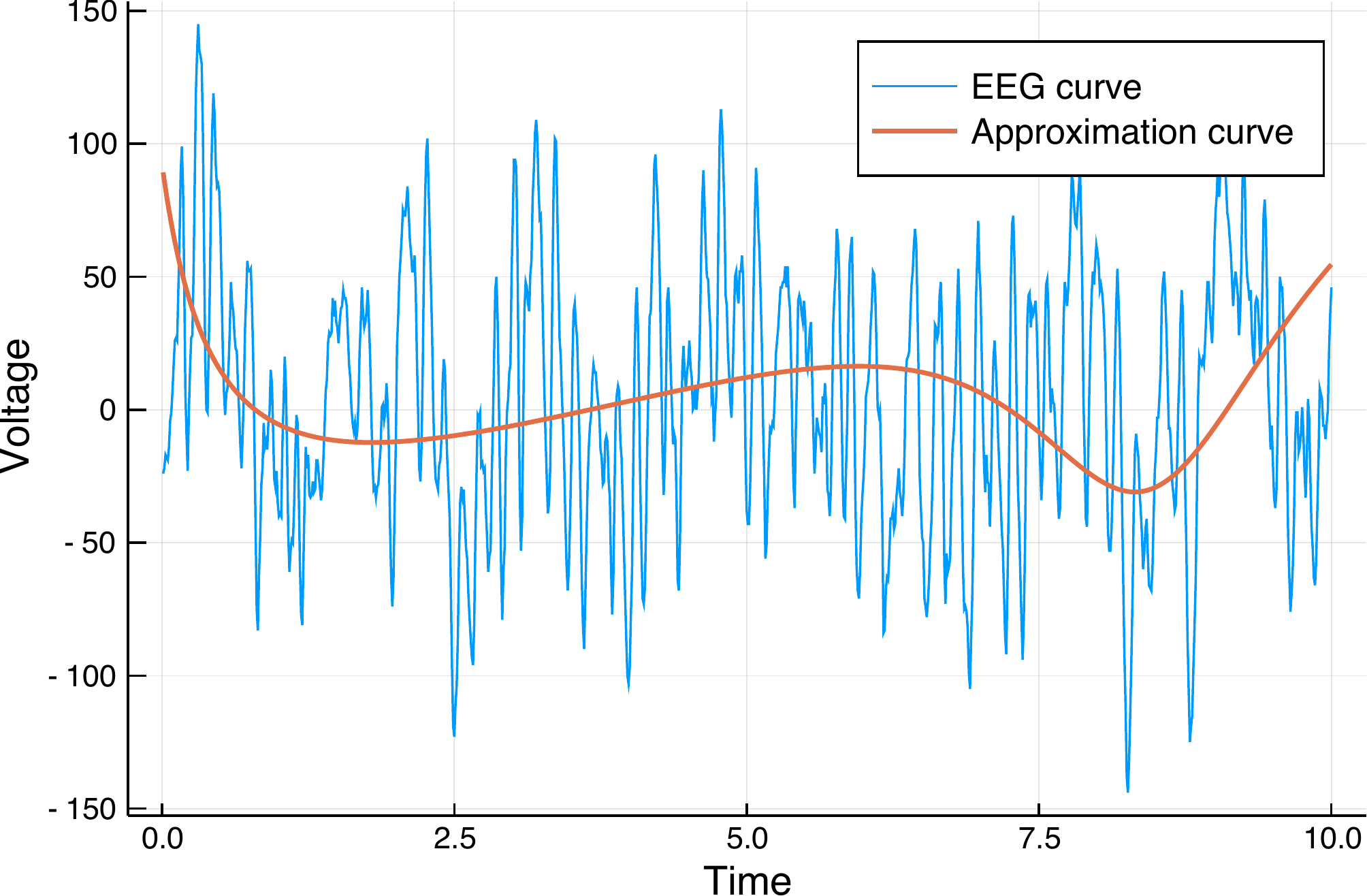} 
			\caption{Model~M1}
			\label{subfig:model1b}
		\end{subfigure}
		\begin{subfigure}{.49\textwidth}
			\includegraphics[width=\textwidth]{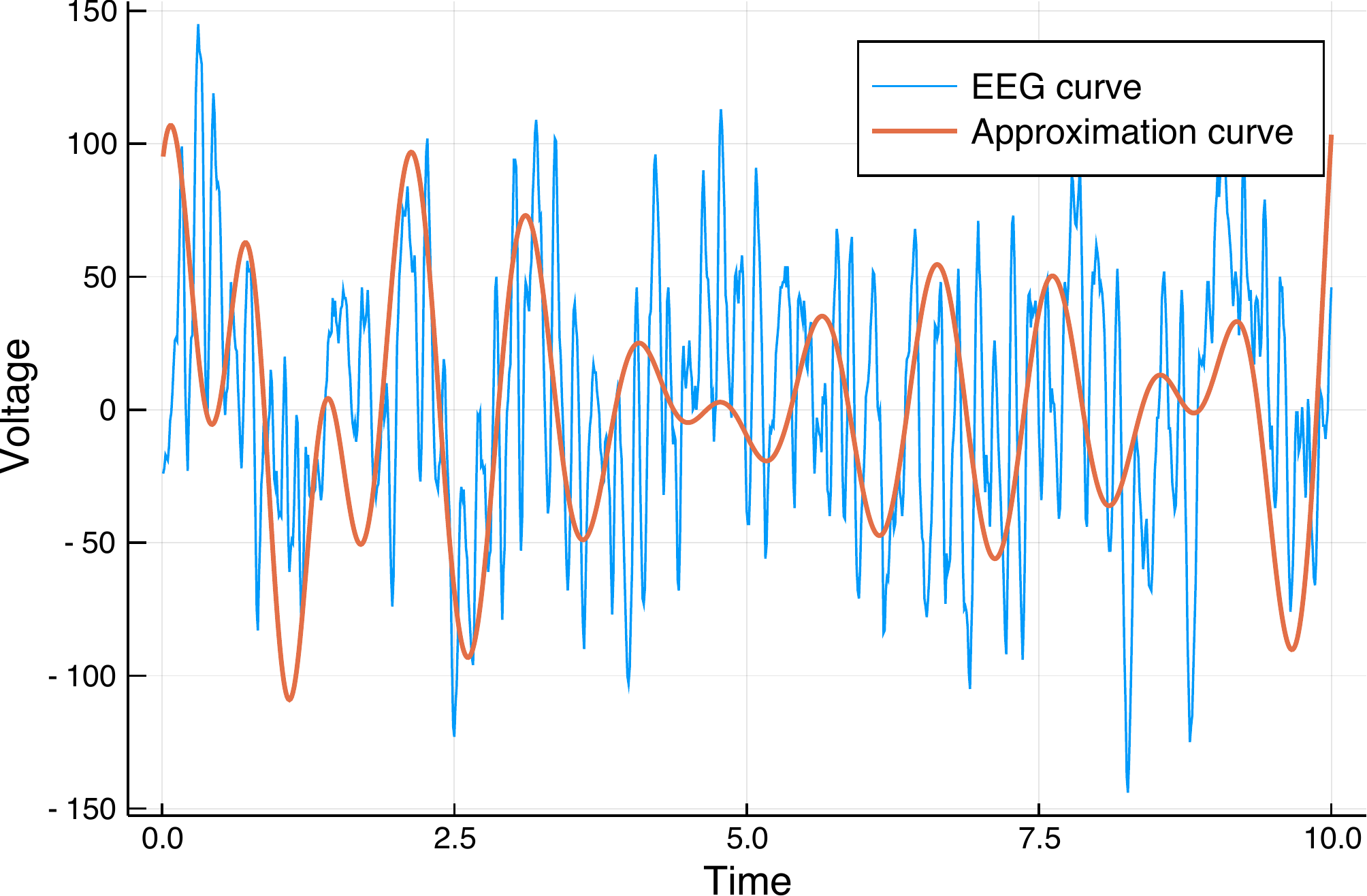}  
			\caption{Model~M2}
			\label{subfig:model2b}
		\end{subfigure}
		\caption{Signal from Class~B and its approximation}
		\label{fig:class2}
	\end{center}
\end{figure}

\subsection{Results}\label{ssec:Results}

In our experiments, where possible,  we use the default parameters from package scikit-learn (0.22.1)~\cite{scikit_learn}. In particular, the activation function is ReLU (unless specified otherwise). The default number of nodes within each hidden layer is~100, but the best results are achieved with fewer nodes. The degree of the numerator and denominator polynomials are $n=3$ and $m=1$. This is where  the highest classification results are achieved. For normalisation purposes, we fixed the value for the zero degree coefficient in the denominator at~$1$. Therefore, the total number of output parameters (features) in Model~M1 is 
$$(n+1)+(m+1)-1=5.$$
These features include the total number of the parameters for both polynomials minus 1 (since one of the coefficients is fixed at~1). In the case of Model~M2 there is an extra parameter that represents the frequency ($\omega$).

\begin{table}
\centering
\caption{One hidden layer}
\label{tab:1h}
\begin{tabular}{|cc|cc|cc|}
\hline
\multicolumn{2}{|c|}{Original (raw data)}&\multicolumn{2}{|c|}{	Model M1}   &\multicolumn{2}{|c|}{Model 2}\\		
\multicolumn{2}{|c|}{4097 features}         &\multicolumn{2}{|c|}{	5 features}&\multicolumn{2}{|c|}{6 features}\\	
\hline													
	Number of nodes 	&Accuracy	&	Nodes & 	Accuracy		&	Nodes  &	Accuracy\\
	\hline
	100 (default)&	55\%	&				100 &	75\%	&100 &75\%\\
	1                     &	60\%	&						1&	40\%	&		1&	40\%\\
	2	                    &35\%			&			2      &	\hl{\textbf{90\%}}	& 		2&	60\%\\
	3	&45\%&							3&	60\%	&		3&	45\%\\
	4	&\hl{\textbf{65\%}}	&						4	&60\%			&4	&65\%\\
	5	&45\%		&					5	&35\%			&5	&80\%\\
	6	&45\%			&				6	&55\%			&6	&70\%\\
	7	&55\%					&		7	&70\%		&	7	&65\%\\
	8	&55\%						&	8	&65\%			&8	&\hl{\textbf{85\%}}\\
	9	&50\%		&					9	&70\%		&	9	&65\%\\
	10	&55\%			&				10&	45\%	&		10&	45\%\\
	\hline
	2731 (2/3 of inputs)&	60\%&3&	60\%&	4 &	65\%\\
	8194 (2*inputs+1)	&50\%&	11&	75\%&		13&	70\%\\
	\hline
\end{tabular}
\end{table}													

 Tables~\ref{tab:1h} and \ref{tab:2h} represent the classification results for one and two hidden layers, respectively. In the tables we introduce the results for different number of nodes within each level. The last section of Table~\ref{tab:1h} represents two widely used rules for choosing the number of nodes for one hidden layer: 
 \[ 2/3k~{\rm{and}}~2k+1, \]
where $k$ represents the size of the inputs (that is, the number of features) for deep learning.     
		
\begin{table}			
\centering
\caption{Two hidden layers}
\label{tab:2h}				
\begin{tabular}{|c|c|c|c|}
\hline
	Number of nodes&{Original (raw data)}&{	Model 1}   &{Model 2}\\		
in each layer&{4097 features}         &{	5 features}&{6 features}\\	
\hline													
	100, 100&	65\%	&				60\%	& 45\%\\
	2, 1	&60\%	&							60\%	&		60\%\\
	3, 2&	30\%	&							40\%&				60\%\\
	4, 3&	60\%	&							45\%	&			75\%\\
	5, 1	&60\%	&							60\%	&			80\%\\
	5, 4&	45\%	&						75\%	&		50\%\\
	6, 5&	70\%	&							40\%	&		80\%\\
	7, 6&	50\%	&							60\%	&			40\%\\
	8, 2&	60\%	&						\hl{\textbf{90\%}}	&			40\%\\
	8, 7&	50\%	&							55\%	&			50\%\\
	8, 8&	60\%	&							25\%	&			80\%\\
	9, 8&	45\%	&							55\%	&	80\%\\
	10 , 9	&45\%&								75\%&			\hl{\textbf{85\%}}\\
	10, 10&	65\%&							80\%&				80\%\\
	11,11	&55\%&								80\%	&			75\%\\
	12, 11	&45\%&								85\%	&			60\%\\
	12, 12&	50\%&							80\%	&		50\%\\
	13, 12&	\hl{\textbf{75\%}}&							85\%	&			55\%\\
	16, 15	&45\%&								80\%&				40\%\\
	25, 25&	50\%&								85\%&				55\%\\
	\hline
\end{tabular}
	\end{table}		

According to Tables~\ref{tab:1h} and \ref{tab:2h}, the classification accuracy can be significantly improved by applying our rational approximation together with the deep learning techniques. Specifically, in the case of a single hidden layer, the accuracy was improved from 65\% (raw data) to 85\% using Model~M2 with 8 nodes and up to 90\% with Model~M1 and 2 nodes. This accuracy can be further improved to 95\% when the activation function is identity (non-default). These cases are highlighted in Table~\ref{tab:1h}.

In the case of two hidden layers, the accuracy was also improved from 65\% (raw data) to 85\% with Model~M2 with 9 or 10 nodes and to 90\% with Model~M1 having 8 or 2 nodes, see Table~\ref{tab:2h}. The accuracy can be further improved also in this case to 95\% when the activation function is $\tanh$ (non-default) and the number of nodes is 12 and 11.

Overall, Model~M1 (which is a simpler model) performs better from both points of view: classification accuracy and computational time. Similar observations are made in~\cite{Zamir2015}, where simpler models are achieving  better classification accuracy than more complex models.
 
 \section{Conclusions and further research directions}\label{sec:conclusions}
 
In this paper, we propose a simple and efficient approach for approximating continuous functions by ratios of linear combinations of basis functions, whose linear coefficients are the decision variables. Our approach exploits the fact that the corresponding optimisation problems are quasiconvex. In all our numerical experiments, the obtained solutions satisfy necessary and sufficient optimality conditions.

We demonstrated that the proposed approach is suitable for approximating nonsmooth and non-Lipschitz functions and also very effective in application to biomedical signal processing and for improving deep learning-based classifiers. 

In the future, we are planning to improve the computational efficiency of the method by considering different basis functions and also by adopting other newly developed optimisation techniques that are exceptionally efficient for quasiconvex functions. Another exciting research direction is to investigate possibilities for overcoming computational difficulties causing by a non-zero defect.    
 
%\section*{References}
\bibliographystyle{plainnat}

%\bibliographystyle{plainnat}

%\section*{\refname}
%\bibliography{mybib}

\end{document}